\documentclass[11pt]{amsart}
\usepackage[margin=1in]{geometry}
\usepackage{amsmath, amssymb, amsfonts, mathtools}
\usepackage{enumitem}
\usepackage{graphicx}
\usepackage{asymptote}
\usepackage{caption}
\usepackage{dsfont} 
\usepackage{comment}

\usepackage{thmtools}

\newtheorem{theorem}{Theorem}[section]
\newtheorem{proposition}[theorem]{Proposition}
\newtheorem{corollary}[theorem]{Corollary}

\newtheorem{lemma}[theorem]{Lemma}

\theoremstyle{definition}

\usepackage{xcolor}
\usepackage{hyperref}

\definecolor{LightBlue}{rgb}{0,0.8,1} 

\hypersetup{colorlinks=true, citecolor=LightBlue, linkcolor=blue,urlcolor=blue} 

\usepackage[noabbrev,capitalise,nameinlink]{cleveref}
\crefname{conjecture}{Conjecture}{Conjectures}

\DeclareMathOperator{\Star}{\mathsf{Star}}

\title{$e$-positive partitions for chromatic symmetric functions}
\author{Noah Kravitz}
\address{St John's College, Oxford and Mathematical Institute, University of Oxford; St Giles', Oxford OX1 3JP, UK}
\email{noah.kravitz@maths.ox.ac.uk}

\begin{document}

\maketitle

\begin{abstract}
We show that the partitions that always appear with nonnegative $e$-coefficients in chromatic symmetric functions of finite graphs are precisely the hook partitions.
\end{abstract}

\section{Introduction}

\subsection{Main result}

This note treats the natural, but previously overlooked, problem of determining which $e$-coefficients of chromatic symmetric functions of graphs can ever be negative.  The \emph{chromatic symmetric function} of a finite graph $G=(V,E)$ is the formal power series
$$X(G):=\sum_{\chi} \prod_{v \in V} x_{\chi(v)} \in \mathbb{Z}[[x_1, x_2, \ldots]],$$
where $\chi$ ranges over the proper vertex colorings of $G$ with the color set $\{1,2,\ldots \}$.  This object, introduced by Stanley~\cite{stanley} in 1995, is a refinement of the chromatic polynomial.  Indeed, evaluating $X(G)$ with $r$ of the $x_i$'s equal to $1$ and the remaining $x_i$'s equal to $0$ gives the number of proper $r$-colorings of $G$. The manner and extent to which $X(G)$ encodes other combinatorial information about $G$ has been studied extensively over the past three decades; see the overviews in \cite{DFvW,MMW,ST}.

The formal power series $X(G)$ is homogeneous of degree $|G|$, and it is a \emph{symmetric function} in the sense that it is invariant under permutations of the variables $x_1, x_2, \ldots$.  Among the natural $\mathbb{Z}$-bases of the fixed-degree homogeneous symmetric functions, one that has attracted particular attention is the \emph{elementary symmetric basis}.  For each positive integer $m$, set
$$e_m:=\sum_{1 \leq i_1<\cdots<i_m} x_{i_1} \cdots x_{i_m}.$$
For each integer partition $\lambda=(\lambda_1, \ldots, \lambda_r)$, define the \emph{elementary symmetric polynomial}
$$e_\lambda:=e_{\lambda_1} \cdots e_{\lambda_r}.$$
It is well known (see, e.g., \cite[Equation (2.4)]{macd}) that the $e_\lambda$'s, for $\lambda$ ranging over the partitions of size $n$, form a $\mathbb{Q}$-basis for the homogeneous symmetric functions of degree $n$.

It follows that for every graph $G$, there are unique integers $c_\lambda(G)$ (for $\lambda$ of size $|G|$) such that
$$X(G)=\sum_{\lambda}c_\lambda(G) e_\lambda.$$
There has been substantial interest, stemming partly from representation-theoretic motivations, in characterizing the \emph{$e$-positive graphs}, namely, the graphs $G$ such that $c_\lambda(G)\geq 0$ for every partition $\lambda$.  A particular focal point in this direction has been the Stanley--Stembridge Conjecture (see \cite{SS,hik}).  We take the ``dual'' perspective of characterizing the \emph{$e$-positive partitions}, namely, the partitions $\lambda$ such that $c_\lambda(G)\geq 0$ for every graph $G$.  The latter problem, which is new, turns out to have a clean solution.  A \emph{hook partition} is a partition with at most a single part larger than $1$.

\begin{theorem}\label{thm:main}
The $e$-positive partitions are precisely the hook partitions.
\end{theorem}

In other words, this theorem says that the potential obstructions to $e$-positivity of graphs are precisely the non-hook partitions.  We also provide a combinatorial interpretation (see Section~\ref{sec:combinatorial}) of the nonnegativity of $c_\lambda(G)$ when $\lambda$ is a hook partition.

\subsection{Further questions}

Theorem~\ref{thm:main} brings to mind several natural questions for future inquiry.  For each graph $G$, let $$N(G):=\{\lambda: c_\lambda(G)<0\}$$ be the set of partitions with negative $e$-coefficients in $X(G)$.  Theorem~\ref{thm:main} shows that the union of the $N(G)$'s (over all graphs $G$) is the set of non-hook partitions.  An ambitious general problem is characterizing all of the possibilities for $N(G)$.  Short of providing a full characterization, one might start by determining which \emph{pairs} of partitions can appear together in sets $N(G)$.

One could also ask which partitions appear in the sets $N(G)$ when $G$ ranges over a restricted class of graphs.  A natural starting point, at least from a historical perspective, would be the class of claw-free (or claw-contractible-free) graphs.  Indeed, the claw $K_{1,3}$ is the smallest non-$e$-positive graph, and Stanley~\cite{stanley} suggested that contractibility to the claw might be a driver of non-$e$-positivity; this notion was dispelled only recently by Dahlberg, Foley, and van Willigenburg~\cite{DFvW}.

Finally, one could study the same questions for the Schur basis.  It follows from a result of Kaliszewski~\cite{kal} that hook partitions always appear with nonnegative Schur-coefficients in chromatic symmetric functions.  The problem for non-hook partitions remains open.  In connection with the discussion from the previous paragrph, we also mention a well-known conjecture of Gasharov~\cite{gas}, recently disproven by Matherne and Morales~\cite{MM}, that claw-free graphs are Schur-positive.

\section{Proofs}

\subsection{Preliminaries on the power sum basis}
We will make repeated use of a basic result of Stanley on the expansion of chromatic symmetric functions in the power sum basis.

For each positive integer $m$, set $$p_{k}:=\sum_{i \geq 1} x_i^k.$$
For each partition $\lambda=(\lambda_1, \ldots, \lambda_r)$, define the \emph{power sum}
$$p_\lambda:=p_{\lambda_1} \cdots p_{\lambda_r}.$$
The $p_\lambda$'s, for $\lambda$ ranging over the partitions of size $n$, form another $\mathbb{Z}$-basis for the homogeneous symmetric functions of degree $n$.  Newton's identity (see, e.g., \cite[Equation (2.11')]{macd})
\begin{equation}\label{eq:newton}
(-1)^{n-1}p_n=ne_n+\sum_{i=1}^{n-1} (-1)^{i} e_{n-i} p_i
\end{equation}
provides a way to convert from the power sum basis to the elementary symmetric basis.

Let $G=(V,E)$ be a finite graph.  For each subset $S \subseteq E$, let $\lambda(S)$ be the partition (of size $|G|$) that records the sizes of the connected components of the subgraph of $G$ with edge set $S$.  Using an inclusion-exclusion argument, Stanley~\cite[Theorem 2.5]{stanley} showed that
\begin{equation}\label{eq:stanley}
X(G)=\sum_{S \subseteq E}(-1)^{|S|}p_{\lambda(S)};
\end{equation}
this simple fact is surprisingly useful.

\subsection{Non-hook partitions}

To show that non-hook partitions can appear with negative $e$-coefficients, we provide an explicit construction.

\begin{proposition}\label{prop:non-hook}
If $\lambda$ is a non-hook partition, then there is some graph $G$ such that $c_\lambda(G)<0$.
\end{proposition}

Our construction, motivated by the example of the claw, will be the disjoint union of a star graph and some number of complete graphs.  Let $\Star_n:=K_{1,n-1}$ denote the star graph on $n$ vertices.

\begin{lemma}\label{lem:star}
Let $n \geq 4$.  Then $c_{(n-k,k)}(\Star_n)<0$ for every $2 \leq k \leq n/2$.
\end{lemma}

\begin{proof}
It is an immediate consequence of \eqref{eq:stanley} (see \cite[Theorem 8]{CvW} for a detailed explanation) that
\begin{equation*}\label{eq:star-expansion}
X(\Star_n)=\sum_{m=0}^{n-1}(-1)^m \binom{n-1}{m}p_{m+1} p_{1}^{n-m-1}.
\end{equation*}
Recall that each $p_k$ is a linear combination of the functions $e_\lambda$ for $\lambda$'s of size $k$, and that $e_\lambda e_{\lambda'}=e_{\lambda \sqcup \lambda'}$.  Thus the product $p_{m+1}p_{1}^{n-m-1}$ is a linear combination of the functions $e_\lambda$ for $\lambda$'s with at least $n-m-1$ parts of size $1$.  Since we are concerned with only the $e_{(n-k,k)}$-coefficient (where $k \geq 2$), we can restrict our attention to the contribution of the $m=n-1$ term, namely, $(-1)^{n-1}p_n$.

A double application of Newton's identity \eqref{eq:newton} gives
\begin{align*}
(-1)^{n-1}p_n &=ne_n+\sum_{i=1}^{n-1} (-1)^{i} e_{n-i} p_i\\
 &=ne_n-\sum_{i=1}^{n-1} e_{n-i} \left(i e_i+\sum_{j=1}^{i-1} (-1)^j e_{i-j}p_j \right)\\
 &=ne_n-\sum_{i=1}^{n-1} ie_{n-i}e_i-\sum_{1 \leq j<i<n} (-1)^j e_{n-i} e_{i-j}p_j.
\end{align*}
The last sum is a linear combination of the functions $e_\lambda$ for $\lambda$'s with at least $3$ parts, so we may ignore it.  Thus, examining the first sum, we see that the $e_{(n-k,k)}$-coefficient of $X(\Star_n)$ is $-n$ if $k \neq n/2$ (with contributions of $-k$ and $-(n-k)$ from $i=k$ and $i=n-k$) and is $-n/2$ if $k=n/2$ (with just a contribution of $-k$ from $i=k$); either way, we have $c_{(n-k,k)}(\Star_n)<0$.
\end{proof}

The deduction of Proposition~\ref{prop:non-hook} is now quick.

\begin{proof}[Proof of Proposition~\ref{prop:non-hook}]
Let $\lambda=(\lambda_1, \ldots, \lambda_r)$ be a non-hook partition.  Then $r \geq 2$ and $\lambda_2 \geq 2$.  Consider the disjoint union $$G=G(\lambda):=\Star_{\lambda_1+\lambda_2}\sqcup K_{\lambda_3}\sqcup\cdots\sqcup K_{\lambda_r}.$$
Since $X(K_m)=m!e_m$ for every $m \geq 1$, we have
$$c_\lambda(G)=c_{(\lambda_1,\lambda_2)}(\Star_{\lambda_1+\lambda_2}) \cdot \lambda_3! \cdots \lambda_r!,$$
which is negative by Lemma~\ref{lem:star}.
\end{proof}

\subsection{Hook partitions}
The main work is showing that hook partitions always appear with nonnegative $e$-coefficients.

\begin{proposition}\label{prop:hook}
If $\lambda$ is a hook partition, then $c_\lambda(G) \geq 0$ for every graph $G$.
\end{proposition}

We record the following facts about the coefficients of hook partitions in the elementary symmetric basis expansion of a power sum.  For a symmetric function $f$ and a partition $\lambda$, let $[e_\lambda](f)$ denote the coefficient of $e_\lambda$ when $f$ is expressed in the elementary symmetric basis.  Write $(m,1^{n-m})$ for the hook partition with one part of size $m$ and $n-m$ parts of size $1$.

\begin{lemma}\label{lem:hook-in-power}
Let $\ell \geq 1$ and $n \geq m \geq 2$.  Then
$$[e_{(1^\ell)}](p_\ell)=1 \quad \text{and} \quad [e_{(m,1^{n-m})}](p_n)=(-1)^{m-1} n.$$
\end{lemma}

\begin{proof}
The first statement is immediate from the fact that $e_{(1^\ell)}$ is the only elementary symmetric polynomial of degree $\ell$ that contains pure powers $x_i^\ell$.

We prove the second statement by induction on $n-m$.  From Newton's identity~\eqref{eq:newton}, we read off
\begin{align*}
[e_{(m,1^{n-m})}](p_n) &=[e_{(m,1^{n-m})}] \left((-1)^{n-1}ne_n+\sum_{i=1}^{n-1} (-1)^{n-1+i} e_{n-i}p_i \right).
\end{align*}
For the base case $m=n$, only the first term contributes, and we obtain the desired identity $$[e_{(m,1^{n-m})}](p_n)=(-1)^{n-1}n.$$
For the induction step, suppose that $m<n$.  Our induction hypothesis gives
\begin{align*}
[e_{(m,1^{n-m})}](p_n) &=[e_{(m,1^{n-m})}] \left((-1)^{m-1}e_m p_{n-m}+e_1 p_{n-1}\right)\\
 &=(-1)^{m-1}[e_{(1^{n-m})}] \left(p_{n-m}\right)+[e_{(m,1^{n-m-1})}]\left(p_{n-1}\right)\\
 &=(-1)^{m-1}+(-1)^{m-1}(n-1)\\
 &=(-1)^{m-1}n,
\end{align*}
as desired (the first equality is where we make crucial use of the condition $m \geq 2$).
\end{proof}

\begin{lemma}\label{lem:hook-in-power-2}
Let $n \geq m \geq 2$, and let $\lambda=(\lambda_1, \ldots, \lambda_r)$ be a partition of size $n$.  Then
$$[e_{(m,1^{n-m})}](p_\lambda)=(-1)^{m-1} \sum_{j: \lambda_j \geq m} \lambda_j.$$
\end{lemma}

\begin{proof}
We have
\begin{equation}\label{eq:hook-in-power}
[e_{(m,1^{n-m})}](p_\lambda)=\sum_{(\mu^{(1)}, \ldots, \mu^{(r)}) \in T}~ \prod_{i=1}^r [e_{\mu^{(i)}}]\left(p_{\lambda_i} \right),
\end{equation}
where $T$ is the set of tuples of partitions $(\mu^{(1)}, \ldots, \mu^{(r)})$ such that each $\mu^{(i)}$ has size $\lambda_i$ and $\mu^{(1)} \sqcup \cdots \sqcup \mu^{(r)}=(m,1^{n-m})$.  The set $T$ admits a simple description: It consists of the tuples $(\mu^{(1)}, \ldots, \mu^{(r)})$ with
$$\mu^{(j)}=(m, 1^{\lambda_j-m}) \quad \text{and} \quad \mu^{(i)}=(1^{\lambda_i}) \text{ for all $i \neq j$},$$
where $j$ is an index with $\lambda_j \geq m$.  By Lemma~\ref{lem:hook-in-power}, each such tuple contributes $(-1)^{m-1}\lambda_j$ to \eqref{eq:hook-in-power}.
\end{proof}

The idea behind the proof of Proposition~\ref{prop:hook} is that Stanley's identity~\eqref{eq:stanley} and Lemma~\ref{lem:hook-in-power-2} lead to an expression for $c_\lambda(G)$ that is amenable to a deletion-contraction argument, which we isolate in advance.  Given a multigraph $G=(V,E)$ (allowing loops and parallel edges)\footnote{The passage to multigraphs is necessary only for our inductive scheme in the following lemma.} and a vertex $v \in V$, let $A(G,v,m)$ denote the set of edge subsets $S \subseteq E$ such that in the subgraph of $G$ with edge set $S$, the connected component of $v$ has size at least $m$.  Define the quantity
$$Q(G,v,m):=(-1)^{m-1}\sum_{S \in A(G,v,m)}(-1)^{|S|}.$$
We will deduce the nonnegativity of $c_\lambda(G)$ from the nonnegativity of the quantities $Q(G,v,m)$.

\begin{lemma}\label{lem:deletion-contraction}
Let $m \geq 1$, and let $G=(V,E)$ be a finite multigraph.  Then for each $v \in V$, we have $Q(G,v,m) \geq 0$.
\end{lemma}

\begin{proof}
We induct on $m$, and for each fixed value of $m$ we induct on $|E|$.  When $m=1$, the set $A(G,v,1)$ consists of all of the $2^{|E|}$ subsets of $E$, and the expression of interest is
\begin{equation}\label{eq:m=1}
\sum_{S \subseteq E}(-1)^{|S|}=\begin{cases}
0, &\text{if } E \neq \emptyset;\\
1, &\text{if } E=\emptyset,
\end{cases}
\end{equation}
which is indeed nonnegative.

We now turn to the inductive step.  Fix some $m\geq 2$, and induct on $|E|$.  Consider the edges incident to $v$.  If $v$ is not incident to any non-loop edges (which, in particular, occurs for the base case $|E|=0$), then $A(G,v,m)$ is empty and the sum in question vanishes.  Now suppose that $v$ is incident to some non-loop edge $e=\{v,u\} \in E$ (so $u \neq v$).  Conditioning on whether or not $S$ contains $e$ leads to the deletion-contraction relation
\begin{align*}
Q(G,v,m) &=(-1)^{m-1}\sum_{e \in S \in A(G,v,m)}(-1)^{|S|}+(-1)^{m-1}\sum_{e \notin S \in A(G,v,m)}(-1)^{|S|}\\
 &=(-1)^{m-1}\sum_{S' \in A(G/e,v,m-1)}(-1)^{|S'|+1}+(-1)^{m-1}\sum_{S \in A(G-e,v,m)}(-1)^{|S|}\\
 &=Q(G/e,v,m-1)+Q(G-e,v,m) \geq 0.
\end{align*}
Here $G-e:=(V,E \setminus \{e\})$ is the graph $G$ with $e$ deleted, and $G/e$ denotes the contraction of $G$ by the edge $e$ (keeping the resulting loops and parallel edges), where we have retained the label $v$ for the new vertex that corresponds to both $v,u$ from $G$.  The nonnegativity of the two quantities on the last line is guaranteed by the induction hypothesis.
\end{proof}

We finally combine the pieces to deduce Proposition~\ref{prop:hook}.

\begin{proof}[Proof of Proposition~\ref{prop:hook}]
Let $\lambda$ be a hook partition of size $n$, and let $G=(V,E)$ be a graph with $n$ vertices.  We first dispose of the edge case $\lambda=(1^n)$.  Here, $c_{(1^n)}(G)$ is equal to the coefficient of $\sum_{i} x_i^n$ in $X(G)$, which is of course nonnegative (since every monomial has nonnegative coefficient in $X(G)$). We henceforth restrict our attention to the case where $\lambda=(m,1^{n-m})$ for some $2 \leq m \leq n$.

From \eqref{eq:stanley} and Lemma~\ref{lem:hook-in-power-2} we calculate
\begin{align}
c_{\lambda}(G)=[e_{(m,1^{n-m})}](G) &=[e_{(m,1^{n-m})}] \left(\sum_{S \subseteq E}(-1)^{|S|}p_{\lambda(S)}\right) \notag\\
 &=(-1)^{m-1}\sum_{S \subseteq E}~ \sum_{j: \lambda_j(S) \geq m} (-1)^{|S|} \lambda_j(S), \label{eq:main-expr}
\end{align}
where $\lambda_j(S)$ denotes the $j$-th part of the partition $\lambda(S)$.  The inner sum counts the vertices whose components have size at least $m$ in the subgraph of $G$ with edge set $S$.  Swapping the order of summation, we can write
\begin{equation}\label{eq:Q's}
c_{\lambda}(G)=\sum_{v \in V}(-1)^{m-1}\sum_{S \in A(G,v,m)}(-1)^{|S|}=\sum_{v \in V} Q(G,v,m),
\end{equation}
and Lemma~\ref{lem:deletion-contraction} tells us that the summand for each $v$ is nonnegative.
\end{proof}

It is \emph{a priori} a bit surprising that the nonnegativity of $c_{(m,1^{n-m})}(G)$ is certified vertex-by-vertex.  
For some motivation, notice that on attempting to use deletion-contraction directly with \eqref{eq:main-expr}, one quickly encounters nonnegative linear combinations of expressions of the form
$$\sum_{v \in V}(-1)^{m_v-1}\sum_{S \in A(G,v,m_v)}(-1)^{|S|}$$
for various multigraphs $G=(V,E)$ and various choices of weights $(m_v)_{v \in V}$. If one suspects that these expressions are nonnegative for many different choices of $(m_v)_{v \in V}$, then it is natural to hope that in fact the summand for each $v$ is nonnegative.  This is precisely what Lemma~\ref{lem:deletion-contraction} achieves.

Theorem~\ref{thm:main} is the joint statement of Propositions~\ref{prop:non-hook} and~\ref{prop:hook}.

\subsection{Combinatorial interpretations of nonnegativity}\label{sec:combinatorial}

The nonnegativity of $Q(G,v,m)$ has a neat combinatorial interpretation, due largely to the kind suggestion of an anonymous referee.  Let $G=(V,E)$ be a finite multigraph with a distinguished vertex $v \in V$.  Obtain the subgraph $\overline G=(\overline{V},E)$ of $G$ by deleting all vertices of $V \setminus \{v\}$  that are isolated (i.e., have no incident edges); note that $\overline{G}$ implicitly depends on $v$.  Let $Z(\overline G,v)$ denote the set of acyclic orientations of $\overline G$ in which $v$ is the unique sink (vertex of out-degree $0$).  Let $s(\omega)$ denote the number of sources (vertices of in-degree $0$) of an acyclic orientation $\omega$.  The following proposition provides a combinatorial interpretation for the quantity $Q(G,v,m)$.

\begin{proposition}\label{prop:comb-interpret}
Let $m \geq 1$, let $G=(V,E)$ be a finite multigraph, and let $v \in V$ be a vertex.  Then we have the formula
\begin{equation}\label{eq:comb-interpret}
Q(G,v,m)=\sum_{\omega \in Z(\overline G,v)} \binom{s(\omega)-1}{|\overline{V}|-m}.
\end{equation}    
\end{proposition}
The right-hand side is patently nonnegative since the weights $\binom{s(\omega)-1}{|\overline{V}|-m}$ are always nonnegative.  Proposition~\ref{prop:comb-interpret} provides the following combinatorial interpretation for the coefficient $c_{(m,1^{n-m})}(G)$ when $m \geq 2$.  Let $Z(G)$ denote the set of acyclic orientations of $G$ with a unique sink.

\begin{corollary}\label{cor:comb}
Let $n \geq m \geq 1$, and let $G=(V,E)$ be a finite multigraph with $n$ vertices.  Assume that $G$ has no isolated vertices.  Then we have the formula
$$c_{(m,1^{n-m})}(G)=\sum_{\omega \in Z(G)}\binom{s(\omega)-1}{n-m}.$$
\end{corollary}

\begin{proof}
For $m=1$, it is straightforward to check that the quantity in question is equal to the indicator function of the event that $G$ has a single vertex and no edges (compare with \eqref{eq:m=1} above and with the first full paragraph of the proof of Proposition~\ref{prop:comb-interpret} below).  For $m \geq 2$, apply Proposition~\ref{prop:comb-interpret} to each summand of \eqref{eq:Q's}, and then note that the reduced graph $\overline{G}$ is equal to $G$ for all $v \in V$, and that $Z(G)=\cup_{v \in V}Z(G,v)$.
\end{proof}

We remark that the $m=|V|$ instance of Proposition~\ref{prop:comb-interpret} recovers a classical result of Greene and Zaslavsky~\cite{GZ} from 1983.  They showed (as stated in \cite[Theorem 1.2(b)]{stanley}) that if $G=(V,E)$ is a finite graph and $v \in V$ is any vertex, then $|Z(G,v)|$ is $(-1)^{|V|-1}$ times the linear coefficient of the chromatic polynomial $\chi_G$.  It is well known (e.g., via inclusion-exclusion) that the linear coefficient of $\chi_G$ is $\sum_{S \subseteq E \text{ spanning}}(-1)^{|S|}$.  The spanning subsets $S$ are precisely the sets in $A(G,v,|V|)$, so the result of Greene and Zaslavsky is equivalent to the assertion that $$|Z(G,v)|=(-1)^{|V|-1}\sum_{S \in A(G,v,|V|)}(-1)^{|S|}=Q(G,v,|V|).$$
If $\overline{G}=G$, then Proposition~\ref{prop:comb-interpret} with $m=|V|$ tells us that
$$Q(G,v,|V|)=\sum_{\omega \in Z(G,v)}\binom{s(\omega)-1}{|V|-|V|}=\sum_{\omega \in Z(G,v)}1=|Z(G,v)|.$$
If $\overline{G}\neq G$, then again Proposition~\ref{prop:comb-interpret} with $m=|V|$ tells us that
$$Q(G,v,|V|)=\sum_{\omega \in Z(\overline{G},v)}\binom{s(\omega)-1}{|\overline{V}|-|V|}=0=|Z(G,v)|,$$
where the second equality uses $|\overline{V}|<|V|$, and the third equality uses the fact that every acyclic orientation of $G$ has multiple sinks (since $G$ has multiple connected components).  Thus one can view Proposition~\ref{prop:comb-interpret} as a generalization of Greene--Zaslavsky.

\begin{proof}[Proof of Proposition~\ref{prop:comb-interpret}]
We roughly follow the inductive scheme from the proof of Lemma~\ref{lem:deletion-contraction}.

Suppose that $m=1$.  If $E=\emptyset$, then $\overline{G}$ has only the single vertex $v$ and no edges, in which case $Z(\overline G,v)$ consists of only the trivial orientation; this orientation has $v$ as its only sink and is counted with weight $\binom{1-1}{1-1}=1$. If $E \neq \emptyset$, then each acyclic orientation in $Z(\overline{G},v)$ has at most $|\overline V|-1$ sources, and hence all of the summands on the right-hand side of \eqref{eq:comb-interpret} vanish.  This agrees with \eqref{eq:m=1}.

Suppose now that $m \geq 2$.  To start, consider the degenerate case where $v$ is not incident to any non-loop edges; we must show that the right-hand side of \eqref{eq:comb-interpret} vanishes (recall that $Q(G,v,m)$ vanishes because $A(G,v,m)$ is empty).  The set $Z(\overline{G},v)$ can be nonempty only when $\overline{G}$ has the single vertex $v$ and no edges.  In this case, $Z(\overline{G},v)$ consists of only the trivial orientation (with $1$ source), and its contribution to the right-hand side of \eqref{eq:comb-interpret} is $\binom{1-1}{1-m}=0$.

It remains to consider the main case where $v$ is incident to some non-loop edge $e=\{v,u\} \in E$.  Recall from the proof of Lemma~\ref{lem:deletion-contraction} that $Q$ satisfies the deletion-contraction relation
$$Q(G,v,m)=Q(G/e,v,m-1)+Q(G-e,v,m).$$
Thus it suffices to establish that the right-hand side of \eqref{eq:comb-interpret} satisfies the corresponding relation
\begin{equation}\label{eq:del-contr-binom}
\sum_{\omega \in Z(\overline G,v)} \binom{s(\omega)-1}{|\overline{V}|-m}=\sum_{\omega \in Z(\overline{G/e},v)} \binom{s(\omega)-1}{|V(\overline{G/e})|-(m-1)}+\sum_{\omega \in Z(\overline{G-e},v)} \binom{s(\omega)-1}{|V(\overline{G-e})|-m}.
\end{equation}
For each orientation $\omega \in Z(\overline{G},v)$, let $\omega/e$ and $\omega-e$ denote the restrictions of $\omega$ to the graphs $\overline{G/e}$ and $\overline{G-e}$ (respectively).  We will show that the contribution of $\omega$ to the left-hand side of \eqref{eq:del-contr-binom} is equal to the combined contributions of $\omega/e$ and $\omega-e$ to the right-hand side of \eqref{eq:del-contr-binom}.  We condition (mainly) on whether or not $u$ is isolated in $\overline{G}-e$, i.e., whether or not $u$ is omitted from $\overline{G-e}$.

First, suppose that $u$ is not isolated in $\overline{G}-e$.  Then $Z(\overline{G/e},v)$ consists of the orientations $\omega'=\omega/e$ for $\omega \in Z(\overline{G},v)$ such that $u$ is a sink of $\omega-e$, and $Z(\overline{G-e},v)$ consists of the orientations $\omega'=\omega-e$ for $\omega \in Z(\overline{G},v)$ such that $u$ is not a sink of $\omega-e$; in each case $\omega$ is uniquely determined by $\omega'$ since $e$ must be oriented from $u$ to $v$.  Fix some $\omega \in Z(\overline{G},v)$.  If $u$ is a sink of $\omega-e$, then $\omega/e \in Z(\overline{G/e},v)$ has $s(\omega/e)=s(\omega)$ (the sources are unchanged), and its contribution to the right-hand side of $\eqref{eq:del-contr-binom}$ is
    $$\binom{s(\omega/e)-1}{(|\overline{V}|-1)-(m-1)}=\binom{s(\omega)-1}{|\overline{V}|-m},$$
    while $\omega-e \notin Z(\overline{G-e})$ does not contribute to the right-hand side of \eqref{eq:del-contr-binom}.
    Likewise, if $u$ is a not sink of $\omega-e$, then $\omega-e \in Z(\overline{G-e},v)$ has $s(\omega-e)=s(\omega)$ (note that $v$ cannot become a source in $\omega-e$ because there is still a directed path from $u$ to $v$), and its contribution to the right-hand side of $\eqref{eq:del-contr-binom}$ is
    $$\binom{s(\omega-e)-1}{|\overline{V}|-m}=\binom{s(\omega)-1}{|\overline{V}|-m},$$
    while $\omega/e \notin Z(\overline{G/e})$ does not contribute to the right-hand side of \eqref{eq:del-contr-binom}.
    Summing over $\omega \in Z(\overline{G},v)$ establishes \eqref{eq:del-contr-binom}.

Second, suppose that $u$ is isolated in $\overline{G}-e$ and $v$ is not.  Then $$Z(\overline{G/e},v)=\{\omega/e: \omega \in Z(\overline{G},v)\} \quad \text{and} \quad Z(\overline{G-e},v)=\{\omega-e: \omega \in Z(\overline{G},v)\}.$$
    For each $\omega \in Z(\overline{G},v)$, we have $s(\omega/e)=s(\omega-e)=s(\omega)-1$ (since $u$ is a source of $\omega$ but not of $\omega/e$ or of $\omega-e$).  Thus the combined contribution of $\omega/e$ and $\omega-e$ to the right-hand side of \eqref{eq:del-contr-binom} is
    $$\binom{s(\omega/e)-1}{(|\overline{V}|-1)-(m-1)}+\binom{s(\omega-e)-1}{(|\overline{V}|-1)-m}=\binom{s(\omega)-2}{|\overline{V}|-m}+\binom{s(\omega)-2}{|\overline{V}|-m-1}=\binom{s(\omega)-1}{|\overline{V}|-m}.$$
    For the use of Pascal's identity in the last equality, it suffices to check that $s(\omega)>1$: Indeed, $u$ is a source of $\omega$, and the non-isolation of $v$ in $\overline{G}-e$ guarantees that $\omega$ has another source.

Finally, suppose that both $u,v$ are isolated in $\overline{G}-e$.  If $e$ is not the only edge of $\overline{G}$, then $Z(\overline{G},v),Z(\overline{G/e},v),Z(\overline{G-e},v)$ are all empty (since every orientation has a sink other than $v$) and there is nothing to show.  If $e$ is the only edge of $\overline{G}$, then $Z(\overline{G},v)$ consists of only the orientation that orients $e$ from $u$ to $v$ (with $1$ source), and $Z(\overline{G/e},v)=Z(\overline{G-e},v)$ consists of only the trivial orientation (with $1$ source).  In this case the desired identity is
    $$\binom{1-1}{2-m}=\binom{1-1}{1-(m-1)}+\binom{1-1}{1-m},$$
    which holds due to our assumption that $m \geq 2$. This completes the proof.
\end{proof}

With hindsight, we can motivate the choice of the weight $\binom{s(\omega)-1}{|\overline{V}|-m}$.  The application of Pascal's identity (in the second main case) dictates that the weight should be a binomial coefficient of the form $\binom{s(\omega)-\alpha}{|\overline{V}|-m-\beta}$ for some integers $\alpha,\beta$.  The values of $\alpha,\beta$ can then be ascertained by examining the base case $m=1$ and the ``exceptional case'' where $e$ is the only edge of $\overline{G}$.

The deletion-contraction relation for $Q(G,v,m)$ suggests the possibility that $Q(G,v,m)$ might be a specialization of a new Tutte-type polynomial for graphs with a ``distinguished'' vertex.  It would be interesting to construct such a polynomial and ascertain if it has other specializations with combinatorial interpretations in terms of rooted objects.

\section*{Acknowledgments and AI usage}

The author was supported in part by a NSF Mathematical Sciences Postdoctoral Research Fellowship under grant DMS-2501336.  I thank Steph van Willigenburg for helpful conversations.  I also thank the anonymous referee for their comments, and in particular their suggestion leading to Section~\ref{sec:combinatorial}.

The proof of Proposition~\ref{prop:hook} is due mostly to ChatGPT 5.5 Pro.

\end{document}